\documentclass[11pt]{article}
\usepackage{amscd, amsmath, amssymb, amsthm}
\usepackage[all,cmtip]{xy}
\usepackage[pagebackref]{hyperref}

\title{Complex varieties with infinite Chow groups modulo 2}
\author{Burt Totaro}
\date{  }

\def\Z{\text{\bf Z}}
\def\Q{\text{\bf Q}}
\def\R{\text{\bf R}}
\def\C{\text{\bf C}}
\def\P{\text{\bf P}}
\def\F{\text{\bf F}}

\def\arrow{\rightarrow}

\def\invlim{\varprojlim}
\def\Gal{\text{Gal}}
\def\L{\overline{L}}
\def\QQ{\overline{\Q}}
\def\FF{\overline{F}}
\def\Spec{\text{Spec}}
\def\im{\text{im}}
\def\E{\overline{E}}

\begin{document}
\maketitle
\newtheorem{theorem}{Theorem}[section]
\newtheorem{corollary}[theorem]{Corollary}
\newtheorem{lemma}[theorem]{Lemma}

\theoremstyle{definition}
\newtheorem{definition}[theorem]{Definition}
\newtheorem{example}[theorem]{Example}

\theoremstyle{remark}
\newtheorem{remark}[theorem]{Remark}

Schoen gave the first examples of smooth complex projective
varieties $X$ and prime numbers $l$ for which the Chow group
of algebraic cycles modulo $l$ is infinite
\cite{Schoen}. In particular, he showed
that this occurs for all prime numbers $l$
with $l\equiv 1\pmod{3}$, with $X$ the product of three copies
of the Fermat cubic curve $x^3+y^3+z^3=0$. This is a fundamental
example, showing how far
motivic cohomology with finite coefficients can be from
etale cohomology, which is finite in this situation.
Nonetheless, the restriction on $l$
was frustrating.

Rosenschon and Srinivas then showed that for a very general principally
polarized complex abelian 3-fold $X$, the Chow group $CH^2(X)/l$ is infinite
for all prime numbers $l$ at least some (unknown) constant $l_0$
\cite{RS}.

In this paper, 
we show that for a very general principally
polarized complex abelian 3-fold $X$, the Chow group $CH^2(X)/l$ is infinite
for {\it all }prime numbers $l$ (Theorem \ref{main}).
In particular, these are the first examples
of smooth complex projective varieties with infinite mod 2 Chow groups.
The prime 2 seemed inaccessible for earlier arguments. The mod 2 result
also implies that the Witt group $W(X)$ of quadratic bundles
is infinite \cite{Parimala}, \cite[Theorem 1.4]{TotaroWitt}.
Again, these are the first complex varieties known to have infinite
Witt group.

The method is flexible,
and much of it should apply to other classes of varieties.
The infiniteness of $CH^2(X)/l$ arises from pulling back
Ceresa cycles, as discussed in section \ref{moduli},
by infinitely many different isogenies. A striking feature of the argument
is that the analysis of Chow groups modulo $l$ for a complex variety $X$
involves the reduction of $X$ to characteristic $l$.

Using products $X\times \P^{n-3}$ for any $n\geq 3$, we have similar
examples in higher dimensions:

\begin{corollary}
For each $n\geq 3$, there is a smooth complex projective $n$-fold $X$
such that $CH^i(X)/l$ is infinite for all $2\leq i\leq n-1$ and all
prime numbers $l$.
\end{corollary}

By taking the product with a very general elliptic curve, we get varieties
for which the subgroup of Chow groups killed by $l$ is infinite.
This uses Schoen's theorem on exterior product maps on Chow
groups \cite[Theorem 0.2]{Schoenproduct}.

\begin{corollary}
For each $n\geq 4$, there is a smooth complex projective $n$-fold $X$
such that $CH^i(X)[l]$ is infinite for all $3\leq i\leq n-1$
and all prime numbers $l$.
\end{corollary}

The bounds in these corollaries are optimal. In particular,
for any smooth complex projective $n$-fold $X$ and any prime
number $l$, the group $CH^i(X)/l$
is finite if $i$ is 0, 1, or $n$, and the $l$-torsion subgroup $CH^i(X)[l]$
is finite if $i$ is 0, 1, 2, or $n$. The harder cases
are finiteness of $CH^n(X)[l]$, by Roitman's theorem
\cite[Theorem 5.1]{Blochbook}, and finiteness
of $CH^2(X)[l]$, by the Merkurjev-Suslin theorem \cite[section 18.4]{MS}.

This work was supported by The Ambrose Monell
Foundation and Friends, via the Institute for Advanced Study,
and by NSF grant DMS-1303105.

\section{Moduli spaces}
\label{moduli}

A property holds for {\it very general }complex points
of a complex variety $S$ if it holds for all points outside
a countable union of lower-dimensional closed subvarieties of $S$.
In particular, we can talk about properties of a very general variety
in an irreducible family of varieties.

For a curve $C$ of genus $\geq 2$ with a rational
point $p$ over a field $k$, the {\it Ceresa cycle }is the 1-cycle
on the Jacobian $J(C)$ given by $C-C^{-}$. Here $C$ is embedded
in $J(C)$ with $p$ mapping to $0$, and $C^{-}$ denotes the image
of that curve by the automorphism $x\mapsto -x$ of the Jacobian.
The Ceresa cycle is homologically trivial, and Ceresa showed
that it is not algebraically equivalent to zero for a very general complex
curve $C$ of genus at least 3 \cite{Ceresa}.
The choice of point $p$ is irrelevant if
we only consider the Ceresa cycle modulo algebraic equivalence. Likewise,
for a curve $C$
over an algebraically closed field and a positive integer $m$,
the choice of $p$ does not affect the Ceresa cycle
in $CH_1(J(C))/m$, since the group of cycles algebraically equivalent
to zero is divisible.

For a positive integer $N$,
let $\zeta_N$ be a fixed primitive $N$th root of unity.
Define a (full) {\it level N structure }on
a principally polarized abelian variety $A$ of dimension $g$ to be a basis
$\{u_1,\ldots,u_g,v_1\ldots,v_g\}$ of the subgroup of $A$ killed by $N$
such that, with respect
to the Weil pairing $A[N]\times A[N]\arrow \mu_N$,
we have $\langle u_i,v_i\rangle=\zeta_N$ for all $i$,
$0=\langle u_i,u_j\rangle=\langle v_i,v_j\rangle$
for all $i$ and $j$, and $\langle u_i,v_j\rangle=0$ if $i\neq j$.

Fix a prime number $l$.
Let $N$ be a prime number at least 3 and different from $l$.
Let $X(N)$ be the moduli space of principally polarized abelian
varieties of dimension 3 with a full level $N$ structure with respect
to $\zeta_N$. Then $X(N)$ is a smooth quasi-projective
integral scheme over $\Z[1/N,\zeta_N]$. 

Let $L$ be the function field over $\Q(\zeta_N)$ of the moduli space $X(N)$,
and let $A$ be the natural abelian variety
over $L$. The main theorem will be that $CH^2(A_{\L})/l$ is infinite.
(We need $N\geq 3$ for $L$ and $A$ to make sense, because the moduli
stack ${\cal X}(N)$ has nontrivial generic stabilizer when $N$ is 1 or 2.
Also, note that the algebraic closure $\L$ and the abelian variety $A_{\L}$
are actually independent of the choice of $N$, up to isomorphism.)

By Lecomte and Suslin, for any variety $X$ over an algebraically closed field
$F$ and any algebraically closed extension field $E$ of $F$,
the natural map $CH^2(X)/m\arrow CH^2(X_E)/m$
is an isomorphism \cite{Lecomte, SuslinICM}. As a result,
showing that $CH^2(A_{\L})/l$ is infinite will imply that
$CH^2(A)/l$ is infinite for a very general principally polarized
complex abelian 3-fold $A$.

Let $M=M(N)$ be the moduli space
of curves of genus 3 with a full level $N$ structure on the Jacobian.
Then $M$ is a smooth quasi-projective integral scheme over
$\Z[1/N,\zeta_N]$. The convenient feature of abelian 3-folds for us 
is that the Torelli map $M(N)\arrow X(N)$ is dominant, of degree 2.
(This uses that $N\geq 3$. For $N$ equal to 1 or 2, the moduli
stack ${\cal X}(N)$ has generic stabilizer group of order 2, and the map
$M(N)\arrow X(N)$ of coarse moduli spaces has degree 1.)
That is, most principally polarized
abelian 3-folds $A$ over an algebraically closed field
are Jacobians; but, given a general abelian variety $A$ and the curve $C$,
the isomorphism $J(C)\cong A$ is only determined
up to sign.

Let $E$ be the function field of $M(N)$. For any finite extension
field $E_1$ of $E$ such that the universal curve $C$ over $E$ has an
$E_1$-rational point $p$, we can define the Ceresa cycle
$y\in CH^2(J(C)_{E_1})$. We are usually concerned only with
the class of $y$ in $CH^2(J(C)_{\E})/l^m$ for a natural number $m$;
that class is independent of the choice of $E_1$ and $p$,
since two different Ceresa cycles are algebraically equivalent.
In fact, the same argument shows that $y$ is fixed by the
action of the Galois group $\Gal(\E/E)$ on $CH^2(J(C)_{\E})/l^m$.

Since $E$ is a quadratic extension of $L$, the function field
of $X(N)$, we can view $y$ as a class in $CH^2(A_{\L})/l^m$
for any $m$.
But it is well-defined only up to sign, because of the choice
of isomorphism $J(C)\cong A$. As a result, $\Gal(\L/L)$
acts on $y$ by $gy=y$ if $y$ is in the index-2 subgroup $\Gal(\L/E)$,
and by $gy=-y$ otherwise.

\section{The Ceresa cycle modulo any prime number}

\begin{lemma}
\label{nonzero}
Let $l$ be a prime number, and let $N$ be a prime number
at least 3 and different from $l$.
Let $E$ be the function field over $\Q(\zeta_N)$ of the moduli space
of curves of genus 3 with level $N$ structure,
and let $C$ be the universal curve over $E$.
Let $\L$ be an algebraic closure
of $E$, and let
$y$ be the Ceresa cycle in $CH^2(J(C)_{\L})$ associated
to an $\L$-point of $C$. Then there is a positive integer $c$ such that
$2y$ is not zero in $CH^2(J(C)_{\L})/l^c$.
\end{lemma}

\begin{proof}
The first step is the following result
of Bloch and Esnault \cite[section 1]{BE},
an application of Bloch-Kato's work on $p$-adic Hodge theory.
For a variety $X$ over an algebraically closed field,
the {\it coniveau }filtration on etale cohomology
is defined by: an element $x$ of $H^*(X,\Z/a)$ is in
$N^rH^*(X,\Z/a)$ if there is a closed subset $Y$ of codimension
at least $r$ in $X$ such that $x$ restricts to zero
in $H^*(X-Y,\Z/a)$.

\begin{theorem}
\label{be}
Let $K$ be a field with a discrete valuation $v$, and let $k$
be the residue field. Assume that $K$ has characteristic zero and $k$
is perfect of characteristic $l>0$. Let $X$ be a smooth projective
variety over $K$ with good ordinary reduction at $v$, and let
$Y$ be the special fiber over $k$. Assume either that the crystalline
cohomology of $Y$ has no torsion or that
$$\dim(X)<(l-1)/\gcd(e,l-1),$$
where $e$ is the absolute ramification degree of $K$ (meaning
that $v(K^*)=\Z \cdot (v(l)/e)$).
Finally, let $m$ be a natural number such that $H^0(Y,\Omega^m)\neq 0$.

Then $N^1H^m(X_{\overline{K}},\F_l)\neq H^m(X_{\overline{K}},\F_l)$.
Equivalently, writing $\overline{K}(X)$ for the function field,
the natural map
$$H^m(X_{\overline{K}},\F_l)\arrow H^m(\overline{K}(X),\F_l)$$
is not zero.
\end{theorem}

We will apply Theorem \ref{be} to an abelian variety $X$ with good
ordinary reduction. In this case, the special fiber $Y$
is an ordinary abelian variety over $k$. Every abelian variety
over a perfect field of characteristic $l>0$ has torsion-free crystalline
cohomology \cite[section 7.1]{Illusie}. So Bloch-Esnault's result
applies for all prime numbers $l$ in this case.

Fix a prime number $l$ and a prime number $N\geq 3$ different from $l$.
As in section \ref{moduli}, let $L$ be the function field of the moduli
space $X(N)$ of principally
polarized abelian 3-folds with a level $N$ structure.
Let $A$ be the natural
abelian 3-fold over $L$. (Much of what follows works under some
conditions for other
abelian 3-folds over fields of characteristic zero.)

Let $\Theta \in H^2(A_{\overline{L}},\Q_l(1))$ 
be the given polarization.
The {\it primitive part }$PH^3(A_{\overline{L}},\Q_l(2))$
is the kernel of multiplication by $\Theta$. The hard Lefschetz theorem
over $\C$, translated to etale cohomology,
gives a direct-sum decomposition
$$H^3(A_{\overline{L}},\Q_l(2))=PH^3(A_{\overline{L}},\Q_l(2))
\oplus \Theta\cdot H^1(A_{\overline{L}},\Q_l(1))$$
\cite[p.~122]{GH}.

Let $v$ be the discrete valuation on $L$ whose residue field is the moduli
space of principally polarized abelian 3-folds over $\F_l(\zeta_N)$
with level $N$ structure. Since the generic abelian 3-fold in characteristic
$l$ is ordinary, $A$ has good ordinary reduction
at $v$. By Theorem \ref{be},
$N^1H^3(A_{\overline{L}},\Z/l(2))$ is a proper subgroup of
$H^3(A_{\overline{L}},\Z/l(2))$. (The Tate twist
$\Z/l(2)=(\mu_l)^{\otimes 2}$
makes no difference
to the statement,
since we are considering etale cohomology over an algebraically
closed field.)
It follows that $N^1H^3(A_{\overline{L}},\Z/l^r(2))/l
\arrow H^3(A_{\L},\Z/l^r(2))/l$ is not surjective for any
positive integer $r$.
So the injection
$$B:=(\invlim_r N^1H^3(A_{\L},\Z/l^r(2)))\otimes_{\Z_l}\Q_l
\arrow H^3(A_{\L},\Q_l(2))$$
is not surjective.
(Note that the subspace $B$ may a priori be bigger
than $N^1H^3(A_{\L},\Q_l(2))$. That actually happens
in some examples over $\overline{\F_p}$,
by Schoen \cite[after Theorem 0.4]{Schoenladic}. It would be relatively easy
to prove an upper bound for $N^1H^3(A_{\L},\Q_l(2))$, but we need
Bloch-Esnault's argument in order to prove an upper bound for $B$.)

The Galois group $\Gal(\L/L)$ acts on
$H^3(A_{\L},\Q_l(2))$, preserving the primitive subspace
$PH^3$. The Galois group $\Gal(\L/L\QQ)\subset \Gal(\L/L)$
maps onto a completion of the congruence
subgroup $\Gamma(N)$ of $Sp(6,\Z)$, which acts on $H^1(A,\Q_l)\cong
(\Q_l)^6$ as the standard representation $V$ of
the symplectic group. Since $\Gamma(N)$ is Zariski dense in $Sp(6,\Q_l)$
and $PH^3(A_{\L},\Q_l(2))$ is the irreducible representation $\Lambda^3(V)/V$
of $Sp(6,\Q_l)$, the representation of $\Gal(\L/L)$
on $PH^3(A_{\L},\Q_l(2))$ is irreducible. The Galois group
also preserves the subspace $B$ in the previous paragraph,
and it is clear that $B$ contains the subspace $\Theta\cdot H^1$
(since classes in $\Theta\cdot H^1$ are supported on a theta divisor
in $A_{\L}$). The irreducibility together with the previous
paragraph's result implies that $B$ is {\it equal }to $\Theta\cdot H^1$.

It follows that the inverse limit 
$\invlim_r N^1H^3(A_{\L},\Z/l^r(2))$, a finitely generated
$\Z_l$-submodule of $H^3(A_{\L},\Z_l(2))$, contains
$\Theta\cdot H^1(A_{\L},\Z_l(1))$ as a subgroup of finite index.
So there is an $m\geq 0$ such that for all $r\geq 0$,
$l^m N^1H^3(A_{\L},\Z/l^r(2))$ is contained in $\Theta\cdot
H^1(A_{\L},\Z/l^r(2))$. (In the case at hand (with $\Theta$
a principal polarization), we could take $m=0$, but we choose to state
the argument in a way that would work more generally.)

Let $P\in CH^3(A\times A)$ be a correspondence (with integer
coefficients) such that the action
of $P$ on $H^3(A_{\overline{L}},\Q_l(2))$ sends $\Theta\cdot H^1$
to zero and maps the primitive part $PH^3(A_{\overline{L}},\Q_l(2))$
to itself by an isomorphism. The existence of such a correspondence
is part of the Lefschetz standard conjecture, which is a theorem
for abelian varieties \cite{Kleiman}. (In fact, $P$ can be defined
explicitly as a polynomial in divisor classes
on $A\times A$ \cite[Remark 5.11]{Milne}.) By the previous paragraph,
there is an $a\geq 0$ such that $P_*N^1H^3(A_{\L},\Z/l^r(2))$
is killed by $l^a$ for all $r\geq 0$. 

The Merkurjev-Suslin theorem implies
that that all smooth projective varieties $X$ over $\L$,
Bloch's cycle class map
$$CH^2(X_{\L})[l^{\infty}]\arrow H^3(X_{\L},\Q_l/\Z_l(2))$$
is injective,
with image $N^1H^3(X_{\L},\Q_l/\Z_l(2))$
\cite[section 18.4]{MS}. So the previous
paragraph implies that $P_*(CH^2(A_{\L})[l^{\infty}])$ is killed by
$l^a$.

Let $E$ be the function field of the moduli space $M(N)$ of curves
of genus 3 with level $N$ structure, and let $C$ be the universal
curve over $C$. Then $E$ is a quadratic extension of $L$.
Let $y$ be the Ceresa cycle in $CH^2(A_{E_1})$ associated
to a finite extension $E_1$ of $E$ and an $E_1$-point of $C$,
as in section \ref{moduli}. We are primarily
interested in the image of $y$ in $CH^2(A_{\L})/l^m$ for natural numbers
$m$, which is independent of the choice of $E_1$ and $p$, but which
depends up to sign on the choice of isomorphism $J(C)\cong A_E$.

Let $z=2y$. Then
$z$ is a codimension-2 cycle on $A_{\L}$ which is homologically
trivial, meaning that $z$ maps to zero in $H^4(A_{\L},\Z_l(2))$.
There is an $l$-adic Abel-Jacobi map for homologically trivial
cycles, taking values in continuous
Galois cohomology \cite[section 1]{BST}:
$$CH^2_{\hom}(A_{E_1})\arrow H^1(E_1,H^3(A_{\L},\Z_l(2))).$$
Following Jannsen, continuous cohomology means
the derived functors of the functor $(M_n)\mapsto
\invlim_n (M_n)^G$ on inverse systems \cite{Jannsen}.

We now use that the field $L$ is finitely generated over $\Q$.
The following result is modeled on Bloch and Esnault
\cite[Proof of Proposition 4.1]{BE}.

\begin{lemma}
\label{invariants}
The natural map
$$H^1(L,P_*H^3(A_{\L},\Z_l(2)))\arrow 
H^1(L',P_*H^3(A_{\L},\Z_l(2)))^{\Gal(L'/L)}$$
is an isomorphism for all finite Galois extensions $L'$ of $L$.
\end{lemma}

\begin{proof}
Let $M=P_*H^3(A_{\L},\Z_l(2))$, and let $G=\Gal(\overline{L}/L)$.
Then $M$
is a finitely generated free $\Z_l$-module on which
$G$ acts with nonzero weight $m$ (namely, $m=-1$).
(That is, let $Y$ be a scheme $Y$ of finite type over $\Z$ with fraction
field $L$ (in the case at hand, $Y$ is the moduli space
$X(N)$). To say that $M$ has weight $M$ means
that the eigenvalues of Frobenius on $M\otimes\Q_l$
at all closed points $y$ of $Y$ in some nonempty open subset are algebraic
numbers, with all archimedean absolute values equal to $q^m$, where
$q$ is the order of the residue field at $y$. To prove that,
it suffices to take an
open subset of $Y$ where $A$ has good reduction, and then apply
Deligne's theorem (the Weil conjecture) \cite{Deligne}. Since $A$
is an abelian variety,
we could also reduce to the Weil conjecture for $H^1$,
proved by Weil.)

Let $M_n=M/l^n$ for any natural number $n$.
Write $H^i(G,M)$ for continuous cohomology as defined above.
Since the groups
$M_n$ are finite, the natural map $H^i(G,M)\arrow \invlim_n H^i(G,M_n)$
is an isomorphism for all $i$ \cite[equation 2.1]{Jannsen}.
We want to show that for any open normal subgroup $H$ of $G$,
the natural map
$$H^1(G,M)\arrow H^1(H,M)^{G/H}$$
is an isomorphism.

The Hochschild-Serre spectral sequence gives an exact sequence,
for each $n$:
$$\xymatrix@1{
0\ar[r] & H^1(G/H,M_n^H)\ar[r]& H^1(G,M_n)\ar[r]^-{\alpha_n} &
 H^1(H,M_n)^{G/H} \ar[r] & H^2(G/H,M_n^H).
}$$
The groups on the left are finite, and so they satisfy the Mittag-Leffler
condition as $n$ varies. This implies the exact sequences:
$$0\arrow \invlim_n H^1(G/H,M_n^H) \arrow \invlim_n H^1(G,M_n)
\arrow \invlim_n \im(\alpha_n)\arrow 0$$
and 
$$0\arrow \invlim_n \im(\alpha_n)\arrow \invlim_n H^1(H,M_n)^{G/H}
\arrow \invlim_n H^2(G/H,M_n^H).$$
The groups $M_n^H$ are finite, and so the inverse system $M_n^H$
satisfies Mittag-Leffler. That implies that the continuous
cohomology $H^i(G/H,\invlim_n M_n^H)$ is computed by the complex
of continuous cochains with coefficients in $\invlim_n M_n^H$
\cite[Theorem 2.2]{Jannsen}. But $\invlim M_n^H=(\invlim M_n)^H=0$
because $M$ has nonzero weight as an $H$-module and is torsion free.
So $H^i(G/H, \invlim_n M_n^H)=0$ for all $i$. By the exact sequences
above, the map $H^1(G,M)\arrow H^1(H,M)^{G/H}$ is an isomorphism.
\end{proof}

By section \ref{moduli}, the Ceresa class $y$ and therefore $z=2y$
are invariant under $\Gal(\L/E)$ in $CH^2(A_{\L})/l^m$,
for all natural numbers $m$. By Lemma \ref{invariants},
it follows that $z$ has a well-defined class
in $H^1(E,H^3(A_{\L},\Z_l(2)))$.

Next, we show that $P_*z$ has nonzero image in $H^1(E,H^3(A_{\L},\Q_l(2)))$,
which is defined to mean the continuous cohomology group above
tensored with $\Q_l$ \cite[Definition 5.13]{Jannsen}. This follows
from Hain's proof of Ceresa's theorem. Let $F$ be the direct limit
of the function fields of the moduli spaces $M(N')$ over all positive
integers $N'$. Then $\Gal(F/E)$ is a completion of the congruence subgroup
$\Gamma(N)$ in
$Sp(6,\Z)$. It suffices to show that $P_*z$ in
$H^1(E,PH^3(A_{\L},\Q_l(2)))$ has nonzero restriction to
$H^1(F,PH^3(A_{\L},\Q_l(2)))$. 

The action of the Galois group of $E$
on the cohomology of $A_{\L}$ factors through $\Gal(F/E)$,
and so we are just claiming that $P_*z$ determines a nonzero
homomorphism $\Gal(\L/F)\arrow PH^3(A_{\L},\Q_l(2))$. Here
$\Gal(\L/E\QQ)$ maps onto a completion of the {\it Torelli group},
the kernel of the homomorphism from the genus 3 mapping class group
to $Sp(6,\Z)$. By working over $\C$, it suffices to show that
the Ceresa class determines a nonzero homomorphism from
the Torelli group to $PH^3(A_{\C},\Q)$. (Here the prime
number $l$ is irrelevant.)
This is exactly what Hain's computation of the normal
function of the Ceresa cycle shows \cite[proof of Theorem 8.2]{Hain}.
(In fact, the Ceresa cycle gives an isomorphism from the abelianized
Torelli group tensor $\Q$ to $PH^3(A_{\C},\Q)$. Johnson had earlier
shown that these two groups are isomorphic.)

By the properties
of the correspondence $P$, $P_*z$ takes values
in $H^1(E,P_*H^3(A_{\L},\Q_l(2)))$; so $P_*z$ is nonzero
in $H^1(E,P_*H^3(A_{\L},\Z_l(2)))$. By definition of this
continuous cohomology group, we have an exact sequence \cite[3.1]{Jannsen}
$$0\arrow \invlim_r\nolimits^1 H^0(E,P_*H^3(A_{\L},\Z/l^r(2)))
\arrow H^1(E,P_*H^3(A_{\L},\Z_l(2)))\arrow \invlim_r H^1(E,
P_*H^3(A_{\L},\Z/l^r(2)))\arrow 0.$$
Since $H^3(A_{\L},\Z/l^r(2))$ is finite for each $r$, the $H^0$ groups
on the left are finite,
and so they satisfy the Mittag-Leffler condition as $r$ varies; so the
derived limit $\invlim^1$ is zero. That is,
$$H^1(E,P_*H^3(A_{\L},\Z_l(2)))\cong \invlim_r H^1(E,
P_*H^3(A_{\L},\Z/l^r(2))).$$
It follows that every nonzero element of $H^1(E,P_*H^3(A_{\L},\Z_l(2)))$
is nonzero modulo $l^b$ for some $b\geq 0$. In particular, there is
a $b$ such that
$P_*z$ is nonzero in $H^1(E,P_*H^3(A_{\L},\Z_l(2)))/l^b$.

Assume that there is a cycle $w$ in $CH^2(A_{\L})$ such that
$$l^{a+b}w=z.$$
Since $z$ is homologically trivial and the cohomology of $A_{\L}$
is torsion-free, $w$ is homologically trivial. Let $E'$ be a finite
Galois extension of $E$ such that the cycle $w$ is defined over $E'$,
and consider $w$ as an element of $CH^2_{\hom}(A_{L'})$.
For $\sigma$ in $\Gal(\overline{L}/L)$, we have
\begin{align*}
l^{a+b}P_*(w-\sigma(w))&=P_*(z-\sigma(z))\\
&=0.
\end{align*}
Since $P_*(CH^2(A_{\L})[l^{\infty}])$ is killed by
$l^a$, it follows that $l^aP_*(w-\sigma(w))=0$; that is,
$l^aP_*(w)$ is fixed by $\Gal(\L/E)$. By Lemma \ref{invariants},
it follows that $l^aP_*(w)$ can be viewed as an element $u$
of $H^1(E,P_*H^3(A_{\L},\Z_l(2)))$, and we have
$$l^bu=P_*z$$
in that group. This contradicts
that $P_*z$ is nonzero in $H^1(E,P_*H^3(A_{\L},\Z_l(2)))/l^b$.
Thus there is no element $w$ as above. In other words,
$$z\neq 0 \in CH^2(A_{\L})/l^{a+b}.$$
Since $z$ is 2 times the Ceresa cycle $y$,
Lemma \ref{nonzero} is proved.
\end{proof}

\section{Isogenies}

\begin{theorem}
\label{main}
Let $A$ be a very general principally polarized abelian
3-fold over $\C$. Then $CH^2(A)/l$ is infinite for every
prime number $l$.
\end{theorem}

\begin{proof}
Fix a prime number $N$ at least 3 and different from $l$.
As discussed in section \ref{moduli}, it suffices to show
that $CH^2(A_{\L})/l$ is infinite, where $L$ is the function
field of the moduli space $X(N)$ of principally polarized
abelian 3-folds with level $N$ structure.

We will imitate the strategy Nori used to show that the Griffiths group
tensor $\Q$ has infinite rank for a very general principally
polarized abelian 3-fold $A$ \cite{Nori}. Rosenschon and Srinivas
extended Nori's argument to show that $CH^2(A_{\L})/l$ is infinite
for almost all primes $l$ \cite{RS}.

Namely, $A_{\L}$ is the Jacobian of a curve, and so we have
a Ceresa cycle $y$ on $A_{\L}$, well-defined up to sign
in $CH^2(A_{\L})/l^m$ for any $m$, as discussed in section
\ref{moduli}. By Lemma \ref{nonzero}, there is a positive integer
$c$ such that $z:=2y$ is nonzero in $CH^2(A_{\L})/l^c$.

The plan is to consider infinitely many isogenies from $A$
to other principally
polarized abelian 3-folds. Pulling the Ceresa cycles back
by these isogenies gives
infinitely many nonzero elements of $CH^2(A_{\L})/l^c$. We argue that
these elements of $CH^2(A_{\L})/l^c$
are all different because they all have different actions of the
Galois group $\Gal(\L/L)$. Thus $CH^2(A_{\L})/l^c$ is infinite, and it follows
that $CH^2(A_{\L})/l$ is infinite.

\begin{lemma}
\label{isogeny}
Let $f\colon A\arrow B$ be an isogeny of principally polarized abelian
varieties over an algebraically closed
field $k$. If $f$ has degree prime to $l$, then the pullback
$f^*\colon CH^*(B)/l^c\arrow CH^*(A)/l^c$ is an isomorphism.
\end{lemma}

\begin{proof}
$f_*f^*$ is multiplication by $\deg(f)$, and so $f^*$
is split injective on $CH^*(A)/l^c$. The composition $f_*f^*$
is the sum of the translates by elements of the finite group
$\ker(f)$. These translates
act on Chow groups
as the identity modulo algebraic equivalence. Since $k$
is algebraically closed, the group of cycles algebraically
equivalent to zero is divisible, and so $f_*f^*$ acts as multiplication
by $\deg(g)$ on $CH^*(B)/l^c$. Thus $f^*$ is an isomorphism
on Chow groups modulo $l^c$.
\end{proof}

The abelian 3-fold $A_{\L}$ has many prime-to-$Nl$
isogenies to principally polarized abelian 3-folds.
They are all isomorphic to $A_{\L}$
as schemes (not as schemes over $\L$). By Lemma \ref{isogeny},
the pullback of 2 times the Ceresa cycle $y$ under each of these
isogenies is nonzero in $CH^2(A_{\L})/l^c$. We conclude that
all the pullbacks of the Ceresa
cycle $y$ are not killed by 2 in $CH^2(A_{\L})/l^c$.

It remains to show that for a suitable infinite family
of isogenies, the pullbacks of $y$
are all different in $CH^2(A_{\L})/l^c$. Let $F$ be the direct
limit of the function fields of the moduli spaces $X(M)$
over all positive integers $M$.
Following Nori,
we argue that $\Gal(\L/F)$ acts by different characters
$\Gal(\L/F)\arrow \pm 1$
on all these pullbacks.

Choose a sequence $r_1,r_2,\ldots$ of elements in $Sp(6,\Q)$
which are distinct in the set $Sp(6,\Q)\backslash Sp(6,\Z)$.
We can assume that each $r_i$ is integral (that is, in $Sp(6,\Z_{(p)})$)
at primes $p$ dividing
$Nl$. Just as $Sp(6,\R)$ acts on the Siegel upper half-space,
$Sp(6,\Q)$ acts on the inverse limit of the moduli spaces
$X(M)$.
In particular, $Sp(6,\Q)$ acts by automorphisms on the direct
limit $F$ of the function fields of $X(M)$.

The center $\{\pm 1\}$ of $Sp(6,\Q)$ acts trivially on $F$,
and so we can also think of this as an action $\rho_1$ of $GSp(6,\Q)$
on $F$, with the center $\Q^*$ acting trivially.
Morover, for any element $g\in M_6(\Z)\cap GSp(6,\Q)$,
there are positive integers $a$ and $M$ with a commutative diagram
$$\xymatrix@C-10pt@R-10pt{
A(Ma) \ar[r]\ar[d] & A(M)\ar[d] \\
X(Ma) \ar[r] & X(M),
}$$
where the top map is an isogeny on the fibers. This induces
a commutative diagram
$$\xymatrix@C-10pt@R-10pt{
A_F \ar[r]^{\rho_2(g)}\ar[d] & A_F\ar[d] \\
\Spec(F) \ar[r]^{\rho_1(g)} & \Spec(F).
}$$

For $N\geq 3$, the map $M(N)\arrow X(N)$ has degree 2 and is ramified
over the closure of the image of the divisor
of hyperelliptic curves in $M(N)$. Let $D$ be the corresponding
divisor in the Siegel space $H$. (The level structure
is irrelevant to the definition of $D$; in other words, $D$
is the inverse image of a divisor in the coarse moduli space $X(1)$
of principally polarized abelian 3-folds.)
We use the following observation
by Nori \cite[Lemma]{Nori}:

\begin{lemma}
The subgroup of $Sp(6,\R)$ that maps $D\subset H$ into itself
is equal to $Sp(6,\Z)$.
\end{lemma}

\begin{proof}
The subgroup $K$ of $Sp(6,\R)$ that maps $D$ into itself
is a closed Lie subgroup of $Sp(6,\R)$. The Lie algebra of $K$
is stable under the adjoint action of $K$, which contains
$Sp(6,\Z)$, and that is Zariski dense in $Sp(6,\R)$. So this Lie algebra
is zero or all of $\mathfrak{sp}(6,\R)$. In the latter case,
$K$ is equal to $Sp(6,\R)$, which is false since $D$ is not all of $H$.
So $K$ is discrete. Since $Sp(6,\Z)$ is a maximal discrete
subgroup of $Sp(6,\R)$ \cite[Theorem 7]{Borel},
$K$ is equal to $Sp(6,\Z)$.
\end{proof}

It follows that for any sequence
of elements $g_1,g_2,\ldots$
of $Sp(6,\Q)$ which are distinct in the set $Sp(6,\Q)/Sp(6,\Z)$,
the divisors $g_iD$ in the Siegel space are different. Therefore,
the ramified
double covering of Siegel space pulled back from $M(3)\arrow X(3)$
gives infinitely many non-isomorphic ramified coverings by the 
action of $g_1,g_2,\ldots$. Each of these coverings is pulled
back from a ramified double covering of some finite level $X(N)$. 

For each $i$, $g_i$ of the Ceresa cycle $y$ in $CH^2(A_{\FF})/l^c$
is nonzero, and $y\neq -y$ (since we showed that $2y\neq 0$).
$\Gal(\FF/F)$ acts on that class by
the character $\Gal(\FF/F)\arrow \pm 1$ associated to the translate
by $g_i$ of the quadratic extension of $F$ corresponding
to $M(3)\arrow X(3)$. It follows that these infinitely many 
translates of the Ceresa class are different in $CH^2(A_{\FF})/l^c$.
In particular, $CH^2(A_{\FF})/l^c$ is infinite. It follows
that $CH^2(A_{\FF})/l$ is infinite.
\end{proof}


\small \sc UCLA Mathematics Department, Box 951555,
Los Angeles, CA 90095-1555

totaro@math.ucla.edu
\end{document}